\documentclass[10pt]{amsart}
\usepackage{amssymb,amsbsy,amsmath,amsfonts,times, amscd}
\usepackage{latexsym,euscript,exscale}
\usepackage{graphicx,color} \usepackage{amsthm}
\usepackage{fancyhdr} \usepackage{url}
\usepackage{epigraph}\usepackage{lmodern}
\usepackage{mathrsfs}\usepackage{cancel}
\usepackage[toc,page]{appendix}\usepackage[T1]{fontenc}
\usepackage[all,cmtip]{xy}\usepackage{mathtools}
\usepackage{enumerate}\usepackage{setspace}
\usepackage{helvet}
\usepackage{float}
\usepackage{mathabx} 
\usepackage{hyperref}

\numberwithin{equation}{section}
\newtheorem{theorem}{Theorem}
\newtheorem*{theorem*}{Theorem}

\newtheorem{cor}[theorem]{Corollary}
\newtheorem{lemma}[theorem]{Lemma}
\newtheorem{prop}[theorem]{Proposition}

\theoremstyle{definition}
\newtheorem{defi}[theorem]{Definition}

\theoremstyle{remark}
\newtheorem{remark}[theorem]{Remark}

\numberwithin{theorem}{section}

\newcommand {\Gn}{\langle G \rangle}
\newcommand {\Xs}{X_{sep}}

\newcommand {\Z}{\mathbb Z}

\newcommand{\cB}{\mathcal{B}}

\newcommand{\cA}{\mathcal{A}}

\newcommand{\cT}{\mathcal{T}}

\newcommand{\cO}{\mathcal{O}}

\newcommand{\Coprod}{\text{Coprod}}
\newcommand{\Dqc}{\textbf{D}_{\textbf{qc}}}
\newcommand{\Dcoh}{\textbf{D}_{\textbf{coh}}}
\newcommand{\Dperf}{\textbf{D}_{\textbf{perf}}}

\newcommand{\tn}{{\vert\kern-0.25ex\vert\kern-0.25ex\vert}}
\newcommand{\Tn}{{\big\vert\kern-0.25ex\big\vert\kern-0.25ex\big\vert}}    
\newcommand{\TN}{{\Big\vert\kern-0.25ex\Big\vert\kern-0.25ex\Big\vert}}   

\newcommand\restr[2]{{
  \left.\kern-\nulldelimiterspace 
  #1 
  \vphantom{\big|} 
  \right|_{#2} 
  }}

\usepackage[utf8]{inputenc}                                    
\usepackage[T1]{fontenc}
\usepackage[all,cmtip]{xy}
\usepackage{tikz-cd}

   \DeclareFontEncoding{FMS}{}{}
\DeclareFontSubstitution{FMS}{futm}{m}{n}
\DeclareFontEncoding{FMX}{}{}
\DeclareFontSubstitution{FMX}{futm}{m}{n}
\DeclareSymbolFont{fouriersymbols}{FMS}{futm}{m}{n}
\DeclareSymbolFont{fourierlargesymbols}{FMX}{futm}{m}{n}
\DeclareMathDelimiter{\VERT}{\mathord}{fouriersymbols}{152}{fourierlargesymbols}{147}

\begin{document}

\title[Strong Generators in $\textbf{D}_{perf}(X)$ for schemes with a \textit{separator}]{Strong Generators in $\textbf{D}_{perf}(X)$ for schemes with a \textit{separator}}
\subjclass[2010]{Primary 18E30, secondary 18G20.} 
 \keywords{ Derived Categories, compact generators, schemes, separator }
 
\author{V. B. Jatoba}

\address{Department of Mathematics, Statistics, and Computer Science (M/C 249)\\
University of Illinois at Chicago\\
851 S. Morgan St.\\
Chicago, IL 60607-7045\\
USA}
\email{vjatob2@uic.edu}

\thanks{
The author was partly supported by the Brazilian Federal Agency for the Support and Evaluation of Graduate Education (CAPES), for which he is grateful. }

\date{}
\maketitle

\begin{abstract} This paper extends the result from Amnon Neeman \cite{Amnon} regarding strong generators in $\textbf{D}_{perf}(X)$, from $X$ being a quasicompact, separated scheme to $X$ being quasicompact, quasiseparated scheme that admits a \textit{separator}. Neeman's result states a necessary and sufficient condition for $\textbf{D}_{perf}(X)$ being regular. 

Together with being proper over a noetherian commutative ring, those conditions give an interesting description for when an $R$-linear functor $H$ is representable.
\end{abstract}

\section{Introduction.}\label{SectionIntro}

In 2003, Bondal and Van den Bergh in \cite[Theorem 2.2]{BonVanB}, defined what it means for an object to (strongly) generate a triangulated category. The definitions were inspired by the close relation between certain types of triangulated categories having a strong generator and being saturated, i.e., that every contravariant cohomological functor of finite type to vector spaces is representable. Categories that admit a strong generator were then called \textit{regular}.

Moreover, in the same article Bondal and Van den Bergh showed that whenever $X$ is a smooth variety, $\textbf{D}_{\text{perf}}(X)$ is regular if, and only if, $X$ can be covered by open affine subschemes $Spec(R_i)$ with each $R_i$ of finite global dimension. It was then asked by Bondal and Van den Bergh if one could generalize the condition over the scheme to be quasicompact and separated.

Over the next decade, several steps followed in this direction. First, the case where $X$ is regular and of finite type over a field $k$ was proved by both Orlov \cite{Orlov}, Theorem 3.27, and Rouquier \cite{Rouquier}, Theorem 7.38,. This last paper from Rouquier is also responsible for the generality of the following important theorem.

\begin{theorem}[Rouquier] Let $R$ be a neotherian, commutative ring. Let $\cT$ be a regular triangulated category proper over $R$, and suppose that $\cT$ is idempotent complete. Then an $R$-linear functor $H: \cT \rightarrow R-Mod$ is representable if and only if
\begin{itemize}
\item[i)]H is homological, and 
\item[ii)]for any object $X \in \cT$, the direct sum $\oplus^{\infty}_{i=-\infty} H(\Sigma^i X)$ is a finite $R$-module.
\end{itemize}
\end{theorem} 

This motivate finding examples of regular, idempotent complete triangulated categories proper over a noetherian ring $R$. In particular, it is a well-known fact that the category $\textbf{D}_{\text{perf}}(X)$, for $X$ a quasicompact, quasiseparated scheme, is idempotent complete.

In 2017, Neeman \cite{Amnon} proved the Bondal and Van den Bergh conjecture, i.e.

\begin{theorem}[Neeman] \label{IntroTheo1} Let $X$ be a quasicompact, separated scheme. Then $\textbf{D}_{\text{perf}}(X)$ is regular if, and only if, $X$ can be covered by open affine subschemes $Spec(R_i)$, with each $R_i$ of finite global dimension.
\end{theorem}

\begin{remark}\label{rmk1.3}
One direction of the Theorem \ref{IntroTheo1} has been proven in full generality. If $\textbf{D}_{\text{perf}}(X)$ is regular, one may show that if $U = Spec(R)$ is any open affine subscheme of $X$, then $R$ is of finite global dimension. This claim follows by Thomason and Trobaugh \cite{ThTro}, which shows that the restriction functor $j^*: \textbf{D}_{\text{perf}}(X) \to \textbf{D}_{\text{perf}}(U)$ is the idempotent completion of the Verdier quotient map. If $G \in \textbf{D}_{\text{perf}}(X)$ is a strong generator, then $j^* G \in \textbf{D}_{\text{perf}}(U)$ is also a strong generator. By \cite[Theorem 7.25]{Rouquier}, this implies that $R$ must be of finite global dimension. 
\end{remark}

One might ask if the separated condition could be weakened to quasiseparated. As shown above, one of the main applications involves idempotent complete triangulated categories, and $\textbf{D}_{\text{perf}}(X)$ is an idempotent complete triangulated category, for $X$ a quasicompact, quasiseparated scheme. 

This paper gives one step in this direction, extending Theorem \ref{IntroTheo1}. We show that for quasicompact, quasiseparated schemes that admit a \textit{separator,} the theorem hold.

\begin{theorem} \label{main} Let $X$ be a quasicompact, quasiseparated scheme that admits a separator. Then $\textbf{D}_{\text{perf}}(X)$ is regular if, and only if, $X$ can be covered by open affine subschemes $Spec(R_i)$ with each $R_i$ of finite global dimension.
\end{theorem}

A separator is a morphism with some universal property from a quasicompact, quasiseparated scheme to a particular quasicompact separated scheme, introduced by Ferrand and Khan in \cite{Separator}.

One direction of the proof of theorem \ref{main} is identical to the Remark \ref{rmk1.3}, so it remains to show that $\textbf{D}_{\text{perf}}(X)$ is regular if $X$ can be covered by affines of finite global dimension. With the assumption of the existence of a $separator$, the main idea is to pull back the strong generator from the separated scheme and showing that it is again a strong generator in $\Dqc(X)$. Not all quasiseparated schemes admits a separator, but several examples may be found in \cite{Separator}.


\section{Preliminaries.}\label{SecBackground}

\subsection{Local Isomorphism}\label{SubSecLocIso} 

This section follows \cite[Apendix A]{Separator}. Another main source, as usual, is \cite[§4.4, §5]{EGA I}. This is an important property to introduce because the separator is a local isomorphism, which turns out to simplify the proof of the main theorem \ref{thm3.5} in Section 3. 


\begin{defi} [\textbf{Local Isomorphism}] A morphism of schemes $f: X \to Y$ is a \textit{local isomorphism} if every point of $X$ is contained in an open $U \subset X$ such that $f$ induces an open immersion $U \to Y$.
\end{defi}

Local isomorphisms translate the idea of gluing open sets or covering spaces. If $U, V$ are opens of $X$ such that $f$ induces an open immersion, the image of $f( U \cup V)$ is obtained by gluing $f(U)$ and $f(V)$ --- which are isomorphic to $U$ and $V$ respectively --- along the open set $f(U) \cap f(V)$, which contains $f(U \cap V)$.

A local isomorphism is necessarily open, flat and locally of finite presentation \cite[6.2.1]{EGA I}. From the local property of the morphism it also follows that for any point $x \in X$, the induced morphism, $\theta_x \colon \cO_{Y,f(x)} \to \cO_{X,x}$ is an isomorphism. The reverse direction also holds: if $f$ is locally of finite presentation, and $\theta_x$ is an isomorphism for all $x \in X$, then $f$ is a local isomorphism \cite[6.6.4]{EGA I}. 

\begin{prop}
\label{LocIso1}
\cite[Proposition A.3.1]{Separator}
Let $f: X \to Y$ be a separated local isomorphism. If $f$ induces an injection over all maximum points of $X$, then $f$ is an open immersion.
\end{prop}

\subsection{Separator.}\label{SubSecSep} 

A separator of a morphism $f: T \to S$ is another morphism $h$, which is universal among morphisms from $T$ to separated $S$-schemes $E$. This section will follow \cite{Separator}, which contains in-depth explanations and further properties of separators and local isomorphisms. 

\begin{defi}
\label{Def2.1}
Let $f:T \to S$ a morphism of schemes. A \textit{separator} of $f$, or a \textit{separator} through $f$, is a morphism of $S$-schemes $h: T \to E$, with $E$ separated over $S$, such that the following propreties are satisfied:
\begin{itemize}
\item[i)] $h$ is a quasicompact, quasiseparated, surjective local isomorphism, and
\item[ii)] the diagonal morphism $\Delta_h$ is schematic dominant.
\end{itemize}
\end{defi}

If $S = Spec(\Z)$, we call $h$ a separator of $T$. We observe that

\begin{itemize}
\item[i)] A morphism $f \colon Y \rightarrow X$ is \textit{schematic dominant} if $\cO_X \rightarrow f_*(\cO_Y)$ is injective.
\item[ii)] A morphism $f : T \to S$ that admits a separator is quasiseparated. This follows from the fact that the diagonal $\Delta_f$ factorize as
\begin{align*}
 T \xrightarrow{\Delta_h} T \times_E T \xrightarrow{u} T \times_S T
\end{align*}
where $u$ is the induced morphism from the base change. Since $\Delta_h$ is quasicompact, as $h$ is quasiseparated and $u$ is a closed immersion, the composition is quasicompact.
\item[iii)] If $T$ is integral, property $i)$ of Definition \ref{Def2.1} implies property $ii)$ of Definition \ref{Def2.1}.
\end{itemize}

The separator has several desired properties, all of which have an in-depth explanation in \cite{Separator}. We are only interested in the following:

\begin{prop} \label{hopenimm} Let $f: T \to S$ be a morphism and $T \xrightarrow{h} E \xrightarrow{g} S$ a separator of $f$.
\begin{itemize}
\item[i)] Let $U$ be an open set of $T$ that is separated over $S$. Then the restriction of $h$ induces an isomorphism of $U$ to $h(U)$. In particular, $h(U)$ is open and if $T$ is already separated, $h$ is an isomorphism.
\item[ii)](Universal Property) For all $S$-morphisms $h': T \to E'$ with $E'$ separated over $S$, there exists a unique $S$-morphism $u : E \to E'$ such that $h'=uh$.
\end{itemize}
\end{prop}

\begin{proof}
\
\begin{itemize}

\item[i)]
Let $U$ be an open set of $T$ that is separated over $S$. First notice the morphism $U \to E$ induced from $h$ is a separated morphism, since both $U$ and $E$ are, and $h$ is quasi-separated. Second, by the definition of a separator, $\Delta_h$ is schematicaly dominant. By \cite[2.2.1]{Separator}, this implies that the restriction of $h$ to all maximal points of $T$ is injective. Hence by Proposition \ref{LocIso1}, $h$ is an open immersion.

\item[ii)]
Let $h':T \to E'$ be an $S$-morphism with $E'$ separated over $S$. Then, there exists a commutative diagram

$$
\begin{tikzcd}[column sep=5pc]
  T \arrow{d}[swap]{\Delta_{h'}} \arrow{r}{\Delta_h} & T \times_E T \arrow{d}{\phi} \arrow[dashrightarrow]{dl}{\exists w} \\
  T \times_{E'} T \arrow{r}[swap]{{\phi'}} & T \times_S T
\end{tikzcd}
$$
where the morphisms $\phi, \phi'$ are closed immersions, since both $E$ and $E'$ are separated over $S$. Since $\Delta_h$ is schematic dominant by assumption, the conditions on both $\phi$ and $\phi'$,  the requirements for the existence of $w$ are met by the Uniqueness of Schematic Closure \cite[A.5.3]{Separator}.

Hence the diagram 

$$
\begin{tikzcd}[column sep=5pc]
  T \times_E T \arrow{d}[swap]{w} \arrow[r, shift left] \arrow[r, shift right] & T \arrow[d,equal] \arrow{r}{h} & E \arrow[dashrightarrow]{d}{u} \\
  T \times_{E'} T \arrow[r, shift left] \arrow[r, shift right] & T \arrow{r}[swap]{h'} & E'
\end{tikzcd}
$$

commutes, and the result follows.

\end{itemize} \end{proof}

Finally, it is important to understand when a separator exists. The following theorem will give some criteria to work with

\begin{theorem}\label{Thm2.5.1}
Let $f:T \to S$ be a quasi-separated morphism, and let $T_1 \subset T \times_S T$ be the schematic closure of the diagonal morphism $\Delta_f : T \to T \times_S T$. Then, $f$ admits a separator $h$ if, and only if, every irreducible component of $T$ is locally finite ,i.e., every point has an open neighborhood which is disjoint to all but finitely many irreducible components of $T$, over $S$ and both the composition morphisms induced by the projections

\begin{align*}
T_1 \rightarrow & T \times_S T \rightrightarrows T
\end{align*}

are flat and of finite type.

\end{theorem}

The proof is in \cite[Theorem 5.1.1 ]{Separator}

\begin{cor} \label{Cor2.5.3} Let $T$ be a quasiseparated $S$-scheme where each irreducible component is locally finite. Then $T$ admits a separator $h: T \to E$ if, and only if, for all affine opens $U,V$ of $T$, the scheme $U \cup V$ admits a separator. 
\end{cor}

\begin{proof}
It suffices to show that $T$ has a cover by affine opens $U_{\lambda}$ such that the union of every two opens in the cover $U_{\lambda} \cup U_{\mu}$ admits a separator.

First, let $U,V \subset T$ be affine opens. Since $T$ is quasi-separated, the intersection of any affine open with $U \cup V$ is quasi-compact. Recall that a subset $Z$ of a topological space $X$ is said to be \textit{retrocompact} if $Z \cap U$ is quasi-compact for every quasi-compact open subset $U$ of $X$. So $U \cup V$ is retrocompact in $T$.  

Hence, it suffices to show that for all retrocompact open $U \subset T$, $h(U)$ is open and the morphism $U \to h(U)$ is a separator of $U$. That $h(U)$ is open in $E$ follows from the fact that $U$, by hypothesis, is retrocompact. It remains to show that $h(U)$ is separated over $S$. Since $h$ is a local isomorphism, we have the induced isomorphism $h': U \to h(U)$, which induces the commutative diagram 

$$
\begin{tikzcd}[column sep=5pc]
  U \arrow{d}[swap]{i} \arrow{r}{\Delta_{h'}} & U \times_{h(U)} U \arrow{r}{u} & U \times_E U \arrow{d}{i \times i} \\
  T  \arrow{r}[swap]{\Delta_h} & T \times_E T \arrow[r, equal, shift left]  & T \times_E T
\end{tikzcd}
$$
where $u$ is an isomorphism, since $h(U) \to E$ is an immersion. Since $ i \times i$ is an open immersion and $\Delta_h$ is quasi-compact and schematic dominant, $\Delta_{h'}$ is also schematic dominant. Finally, for $h'$ be a separator, it remains to show that it is quasi-compact. But $h'$ can be expressed as the composition of two quasi-compact morphisms, i.e.,

\begin{align*}
U \xrightarrow{i'} h^{-1}(h(U)) \xrightarrow{h} h(U)
\end{align*}
where $i'$ is the open immersion induced by the inclusion $i$. Therefore the condition is necessary.

Next, notice that $(U \cup V) \times (U \cup V) \subset T \times T$ is the union of four canonical opens, namely, $U \times U, V \times V, U \times V, V \times U$. Let $T_1$ be the schematic closure of the diagonal in $T \times T$. Then both $U \times U$ and $V \times V$ are isomorphic via the projection to $U$ and $V$ respectively in $T$, hence flat and of finite type. It suffices to work with $U \times V$. Let $W = T_1 \cap (U \times V)$. By Theorem \ref{Thm2.5.1}, both projections $d_1 : W \to U$ and $d_0 : W \to V$ are flat and of finite type. Since T is quasi-separated, the open immersions $U \to T$ and $V \to T$ are (flat and) of finite type. The open sets $U \times V$, with $U$ and $V$ affines, cover $T \times T$, so the two projections of $T_1$ to $T$ are flat and of finite type and, again, from Theorem \ref{Thm2.5.1}, the corollary follows.
\end{proof}

In \cite{Separator}, Ferrand and Kahn show some schemes that admit a separator and several others that do not. We end this section with some examples:

\begin{itemize}
\item[i)] Every regular locally neotherian scheme of dimension 1 admits a separator, for instance if T is a Neotherian Dedekind scheme over $Spec(\Z)$.
\item[ii)] If $f: T \to S$ is étale of finite presentation and $S$ is normal, then $f$ admits a separator.
\item[iii)] Any normal scheme of finite type over a Noetherian ring admits an open subscheme containing all points of codimension 1and this subscheme has a separator.
\end{itemize}

\subsection{Strong Generators of $\Dqc(\Xs)$}\label{SubSecAmnons}
 
\subsubsection{Strong Generators of a Triangulated Category}

We begin with some definitions, terminology and key properties of a strongly generated category. Most of what is written here follows the first few chapters of \cite{Amnon}.

\begin{defi} Let $\cT$ be a triangulated category and $G \in \cT$ an object. The full subcategory $\Gn_n \subset \cT$ is defined inductively as follows:
\begin{itemize}
\item[i)] $\Gn_1$ is the full subcategory consisting of all direct summands of finite coproducts of suspensions of $G$.
\item[ii)] For $n>1$, $\Gn_n$ is the full subcategory consisting of all objects that are direct summand of an object $y$, where $y$ fits into a triangle $x \to y \to z$, with $x \in \Gn_1$ and $z \in \Gn_{n-1}$.
\end{itemize}
\end{defi}

\begin{defi} Let $G$ be an object in a triangulated category $\cT$. Then $G$ is said to be a \textit{classical generator} if $\cT = \cup_{n=1}^{\infty} \Gn_n$ and a \textit{strong generator} if there exists an $n \in \Z_{\geq 1}$ with $\cT = \Gn_n$. 
\end{defi}

\begin{defi} A triangulated category $\cT$ is called \textit{regular} or \textit{strongly generated} if a strong generator exists.
\end{defi}

\begin{remark} \
\begin{itemize}
\item One might also say that a regular category $\cT$ is built from $G$ in finitely many steps.
\item In \cite{Amnon}, a general discussion about different properties of triangulated category, such as being $proper$ or $idempotent complete$ follows. It also gives insight about the importance of studying such objects.
\end{itemize} 
\end{remark}

\begin{defi} \label{deficoprod} Let $\cT$ be a triangulated category with coproducts, $G \in \cT$ an object and $ A<B$ integers. Then $\overline{\Gn}_{n}^{[A,B]} \subset \cT$ is the full subcategory defined inductively as follows:
\begin{itemize}
\item[i)] $\overline{\Gn}_{1}^{[A,B]}$ is the full subcategory consisting of all direct summands of abritary coproducts of objects in the set $\{ \Sigma^{-i} G, A \leq i \leq B \}$ .
\item[ii)] $\overline{\Gn}_{n}^{[A,B]}$ is the full subcategory consisting of all objects that are direct summand of an object $y$, where $y$ fits into a triangle $x \to y \to z$, with $x \in \overline{\Gn}_{1}^{[A,B]}$ and $z \in \overline{\Gn}_{n-1}^{[A,B]}$.
\end{itemize}
\end{defi}

The difference between the categories $\Gn_n$ and $\overline{\Gn}_{n}^{[A,B]}$ is that $\overline{\Gn}_{n}^{[A,B]}$ allows arbitrary coproducts, but restrict the allowed suspensions to a fixed range from $A$ to $B$.

\subsubsection{Operations between subcategories}

There are several ways to create a new subcategory from others. Some of them will be defined in this section, which follows \cite{Amnon}.

\begin{defi} Let $\cT$ be a triangulated category with $\cA$ and $\cB$ two subcategories of $\cT$. Then:
\begin{itemize}
\item[i)] $\cA \star \cB$ is the full subcategory of all objects $y$ for which there exist a triangle $x \to y \to z$ with $x \in \cA$ and $z \in \cB$.
\item[ii)] $add(\cA)$ is the full subcategory containing all finite coproducts of objects in $\cA$.
\item[iii)] If $\cT$ is closed under coproducts, then $Add(\cA)$ is the full subcategory containing all (set-indexed) coproducts of objects in $\cA$
\item[iv)] If $\cA$ is also a full subcategory, then $smd(\cA)$ is the full subcategory of all direct summands of objects in $\cA$.
\end{itemize}
\end{defi}

Note that the empty coproduct is $0$, hence $0 \in add(\cA) \subset Add(\cA)$ for any $\cA$.

\begin{defi} Let $\cT$ be a triangulated category and $\cA$ a subcategory. Define:
\begin{align*} 
\text{i)} \ coprod_1(\cA) &:= add(\cA);  & coprod_{n+1}(\cA) &:= coprod_1(\cA) \star coprod_n(\cA). \\
\text{ii)} \  Coprod_1(\cA) &:= Add(\cA); & Coprod_{n+1}(\cA) &:= Coprod_1(\cA) \star Coprod_n(\cA).\\
\text{iii)} \  coprod(\cA) &:= \cup^{\infty}_{n=1} coprod_n(\cA).\\
\end{align*}
$\text{iv) } Coprod(\cA) \text{ is the smallest strictly full subcategory of } \cT$ [assumed to have coproducts] containing $\cA$ and satisfying 
\begin{align*}
Add(Coprod(\cA)) \subset Coprod(\cA)  & \ & \text{and}  & \ & Coprod(\cA) \star \Coprod(\cA) \subset Coprod(\cA).
\end{align*}


\end{defi}

\begin{remark} The diagram

$$
\begin{tikzcd}%
        coprod_n(\cA) \arrow[hookrightarrow]{r} \arrow[hookrightarrow]{d} & coprod(\cA)  \arrow[hookrightarrow]{d} 
        \\
        Coprod_n(\cA) \arrow[hookrightarrow]{r} & Coprod(\cA)
    \end{tikzcd}%
$$
commutes. Moreover, the associativity of the $\star$ operation gives that
\begin{align*}
coprod_m(\cA) \star coprod_n(\cA)  = & \  coprod_{m+n}(\cA), \\
Coprod_m(\cA) \star Coprod_n(\cA)  = & \  Coprod_{m+n}(\cA). 
\end{align*}

It can also be shown that $Coprod_1 (Coprod_n (\cA)) = Add (Coprod_n (\cA)) = Coprod_n(\cA)$. Hence $Coprod_n (Coprod_m (\cA)) \subset Coprod_{nm}(\cA)$.

\end{remark}

The following lemma may be found in \cite[Lemma 1.7]{Amnon} and will be used once to prove the next corollary. 

\begin{lemma} Let $\cT$ be a triangulated category with coproducts, $\cT^c$ be the subcategory of compact objects in $\cT$, and let $\cB$ be a subcategory of $\cT^c$. Then
\begin{itemize}
\item[(i)] For $x \in Coprod_n(\cB)$ and $s \in \cT^c$, any map $s \to x$ factors as $s \to b \to x$ with $b \in coprod_n(\cB)$.
\item[(ii)] For $x \in Coprod(\cB)$ and $s \in \cT^c$, any map $s \to x$ factors as $s \to b \to x$ with $b \in coprod(\cB)$.
\end{itemize}
\end{lemma}

\begin{proof}
\cite[Lemma 1.7]{Amnon}
\end{proof}

\begin{cor}
\label{lemma1.8} Let $\cT$ be a triangulated category with coproducts, $\cT^c$ be the subcategory of compact objects in $\cT$, and let $\cB$ be a subcategory of $\cT^c$. Then
\begin{itemize}
\item[(i)] Any compact object in $Coprod_n(\cB)$ belongs to $smd(coprodn(\cB))$.
\item[(ii)] Any compact object in $Coprod(\cB)$ belongs to $smd(coprod(\cB))$.
\end{itemize}
\end{cor}

\begin{proof}
Let $x$ be a compact object in $Coprod_n(\cB)$. The identity map $1 \colon x \to x$ is a morphism from the compact object $x$ to $x \in Coprod_n(\cB)$. By the previous Lemma \ref{Lemma 1.7}, the morphism factors through an object $b \coprod_n(\cB)$. Thus $x$ is a direct summand of $b$ and the results follows.

The same proof holds be removing the subscript $n$, which proves item (ii).
\end{proof}

The next three results follow from these definitions with proofs found in the background section from \cite{Amnon}.

\begin{lemma} \label{Amnlemma1.8} Let $\cT$ be a triangulated category with coproducts, and let $\cB$ be an arbitrary subcategory. Then

\begin{align*}
Coprod_n(\cB) \subset smd(Coprod_n(\cB)) \subset Coprod_{2n} (\cB \cup \Sigma \cB).
\end{align*}
\end{lemma}

\begin{remark} \label{rmkcoprod} Let $\cT$ be a triangulated category with coproducts, and let $\cB \subset \cT$ be a subcategory. For any pair of integers $m \leq n$ define 
\begin{align*}
\cB[m,n] = \bigcup^{-m}_{i=-n} \Sigma^i\cB.
\end{align*}
\end{remark}

\begin{cor} \label{corcoprod} For integers $N >0, A\leq B$ the identity $\overline{\Gn}_N^{[A,B]} = smd(Coprod_N(G[A,B]))$ always holds. Furthermore, one has the inclusions:

\begin{align*}
Coprod_N(G[A,B]) \subset \  \overline{\Gn}_N^{[A,B]}  \subset  Coprod_{2N} (G[A-1,B]).
\end{align*}

\end{cor}

From these two results, one concludes that regarding finiteness conditions there is no loss in generality when working with $Coprod_n(G[A,B])$ instead of $\overline{\Gn}_N^{[A,B]}$, and $smd$ will not change the finiteness of the category generated by $G[A,B]$. So we may work with $Coprod_n(G[A,B])$, which behaves well with $smd$ and the $\star$ operations.

We end this section with a brief discussion about the $Coprod_n(G[A,B])$ subcategory. As stated in [\ref{deficoprod}], $Coprod_n(G[A,B])$ is a full subcategory. Together with Remark \ref{rmkcoprod}, we may let the suspensions free by considering $Coprod_n(G[-\infty, \infty])$. This means that any object $x \in Coprod_n(G[-\infty, \infty])$ factors through $Coprod_n(G[A, B])$ for integers $A$ and $B$. As usual, is it possible that $\cT = Coprod_n(G[-\infty, \infty])$, which motivates the following definition.

\begin{defi} Let $\cT$ be a triangulated category with coproducts and $G \in \cT$ an object in $\cT$. Then, $\cT$ is said to be \textit{fast generated by $G$} if $\cT = Coprod_n(G[-\infty, \infty])$.
\end{defi} 

When $G$ is a compact object, Corollary \ref{corcoprod} tells us that if $\cT = \overline{\Gn}_{n}^{[-\infty,\infty]}$, then $\cT$ is fast generated. In this paper, we will always consider the case when $G$ is a compact generator.

\subsection{The $\Dqc(\Xs)$ case}

Let $\Xs$ be a quasicompact separated scheme. One may consider the category $\Dqc(\Xs)$, which is the unbounded derived category of cochain complexes sheaves of $\cO_{\Xs}-$modules with quasicoherent cohomology, and let $\textbf{D}_{\textbf{perf}}(X)$ be the subcategory of compact objects.

Although the main result is about $\textbf{D}_{\textbf{perf}}(X)$, the next result is the reason why we may work over the bigger triangulated category $\Dqc(X)$, which contains coproducts for $X$ quasicompact, quasiseparated and moreover is compactly generated.

\begin{prop} \label{propreduc}
Let $X$ be a quasicompact, quasiseparated scheme and $G \in \textbf{D}_{\textbf{perf}}(X)$ be a compact generator of $\Dqc(X)$. If $\Dqc(X)$ is fast generated by $G$, then $G$ strongly generates $\textbf{D}_{\textbf{perf}}(X).$
\end{prop}

\begin{proof}
Consider $\cB = \{ \Sigma^i G, i \in \Z \}$. Then Corollary \ref{lemma1.8} gives that 
$\textbf{D}_{\textbf{perf}}(X) = smd(coprod_n(\cB))$, which implies that $G$ strongly generates $\textbf{D}_{\textbf{perf}}(X)$. 
\end{proof}
The path should be clear by now. With the conditions of Theorem \ref{main}, if we show that $\Dqc(X)$ is fast generated by a compact generator, then by the above Proposition \ref{propreduc} the main result will follow.

We finish this section with two more results from \cite{Amnon} stated without proof. 

\begin{theorem}[Neeman]
\label{Amn6.2}
Let $j: V \to \Xs$ be an open immersion of quasicompact, separated schemes, and let $G$ be a compact generator for $\Dqc(\Xs)$. If $H$ is any compact object of $\Dqc(V)$, and we are given integers $n,a \leq b$, then there exist integers $N, A \leq B$ so that $\Coprod _n (\textbf{R}j_*H[a,b]) \subset \Coprod _N(G[A,B])$.
\end{theorem}
\begin{proof}
\cite[Theorem 6.2]{Amnon}
\end{proof}

\begin{theorem}[Neeman]
\label{Amnthm}
Let $\Xs$ be a quasicompact separated scheme. If $\Xs$ can be covered by affine subschemes $Spec(R_i)$ with each $R_i$ of finite global dimension, then there exists a compact generator $G$  that fast generates $\Dqc(\Xs)$, i.e., $\Dqc(\Xs) = Coprod_n(G[-\infty, \infty])$.
\end{theorem}

\begin{proof}
\cite[Theorem 2.1]{Amnon}
\end{proof}

\section{Schemes with Separator}

Throughout this section, assume $X$ to be a quasicompact, quasiseparated scheme with separator $f: X \to \Xs$. Without lost of generality, $X$ may be written as $X = U \cup V$ with $U$ and $V$ quasicompact open subschemes of $X$. Let $V$ to be affine and $Z$ be the closed complement of $U$ on $X$, i.e., $Z = X \backslash U$. Then $Z \subset V$ and we have the commuting diagram

$$
\begin{tikzcd}[column sep=5pc]
  Z \arrow{d}{c} \\
  V \arrow{d}{i} \arrow{rd}{j} & { } \\
  X \arrow{r}{f} & \Xs
\end{tikzcd}
$$
where $c: Z \to V$ is a closed immersion and $i: V \to X$, $j: V \to \Xs$ are open immersions.

\begin{remark}
Throughout this section, the index $[A,B]$ is omitted, as the range itself is not relevant for almost all proofs, only that it is finite. That means that $\Coprod _N(G[A,B])$ for some integers $A < B$ is written as $\Coprod _N(G)$.
Unless otherwise specified, $G$ will be the compact strong generator of $\Dqc(\Xs)$, which exists by Theorem \ref{Amnthm}.
\end{remark}

\begin{lemma}
\label{lem3.1}
Assume $\Xs$ to be a quasicompact, separated scheme and $V \subset \Xs$ an open subscheme. For $P \in \Dperf(V )$, let $j: V \to \Xs$ be the open immersion and $G$ the compact strong generator of $\Dqc(\Xs)$ . Then the pushforward $j_*P$ is in $\text{Coprod}_N(G)$.
\end{lemma}

\begin{proof}
Notice that $V$ and $\Xs$ are separated and since $P$ is in $\Coprod _M(j^*G)$, for $G$ is a global generator, it follows from Theorem \ref{Amn6.2} that $j_*(P) \in \Coprod _N(G)$.
\end{proof}

\begin{prop}
\label{prop3.2}
Let $P \in \Dqc(V \text{on } Z)$, $i: V \to X$ and $j: V \to \Xs$ be the open immersions. Then $i_*P$ is a retract of $f^*j_*P$.
\end{prop}

\begin{proof}

First, we show that $f^*j_*P$ is supported on $ Z \coprod W$, by viewing $Z$ as $Z = X \backslash U$ and $W$ as some closed subset of $U \subset X$. Now it suffices to show that the pullback of  $f^*j_*P$ to the intersection $U \cap V$ is zero. That would imply that there exist a closed subset in $X \backslash V \subset U$, say $W$, that contains the remainder (if any) of the  pullback via $f$ of $j_*P$ that is not in $Z$.

Since the diagram

$$
\begin{tikzcd}[column sep=5pc]
  {} & V \arrow{d}{i} \arrow{rd}{j} & { } \\
  U \cap V \arrow{ru}{k} \arrow{r}{l} &
  X \arrow{r}{f} & \Xs
\end{tikzcd}
$$
is commutative (every map, except $f$ is an open immersion), $Z \nsubseteq U \cap V$, using the counit equivalence map one can see that

\begin{align*}
l^*f^*j_*P &\cong k^*j^*j_*P \\
		   &\cong k^*P \\
		   &\cong 0.
\end{align*}

Hence, $f^*j_*P \simeq R \oplus S$, with $S$ supported on $W \subset U$ and $R$ supported on $Z \subset V$. Finally, since $V$ is separated and $f$ is a local isomorphism on separated open subschemes, the restriction of $f^*j_*P$ to $V$ is $i_*P$, i.e., $i_*P \simeq R$.

Therefore $f^*j_*P \simeq i_*P \oplus Q$ and the result follows.\end{proof} 
The next goal is to show that the pullback of a compact generator via a separator is again a compact generator. This will be done in several steps. First, notice that the isomorphism over separated opens property from the separator induces locally the notion of ``\textit{non separated points}''. Those are the points in an open affine for which the separator is not an isomorphism.

\begin{defi}Let $f: X \to Y$ be the separator and $V \subset X$ be an open affine. Define $Z_V$ as the closure of the set $\{x \in V; f^{-1}f(x) \neq \{x\}\}$ (the \textit{non separated points of V}).
\end{defi}

$Z_V$ is then a closed subscheme of $V$.




Recall that the localization sequence for $\Dqc(X)$ holds true for $X$ quasicompact and quasiseparated, i.e, for $U \subset X$ quasicompact open and $Z$ the closed complement, one have

\begin{align*}
\Dqc(X \ \textbf{on} \ Z) \rightarrow \Dqc(X) \rightarrow \Dqc(U)
\end{align*}

Moreover, $G \in \Dqc(X)$ is a compact generator if, and only if, for any $F \in \Dqc(X)$, $Hom_{\Dqc(X)} ( G, F ) = 0$ implies $F = 0$.

With that in mind, it is possible to prove that the pullback via a separator of a compact generator is a compact generator.

\begin{prop}
\label{lem3.3}
Let $f: X \to Y$ be the separator. If $G \in \textbf{D}_{\textbf{perf}} (Y)$ is a compact generator, then $f^*G \in \textbf{D}_{\textbf{perf}} (X)$ is a compact generator.
\end{prop}

\begin{proof}


Since $X$ is quasicompact and quasiseparated, it suffices to show that the restriction of $f^*(G)$ is a generator for any affine open and for any quasicompact open subscheme of $X$.

For the affine case consider the commutative diagram 

 \[
\begin{tikzcd}
 &  X \arrow{dr}{f} &  \\
 V \arrow[hookrightarrow]{ur}{i} \arrow{r}{Id} &  V \arrow[hookrightarrow]{r}[swap]{\restr{f}{V}}& Y   
\end{tikzcd}
\]
where $V \hookrightarrow X$ is an open affine.

The restriction to $V$ is indeed a generator, since $\restr{f}{V}$ is an isomorphism on $f(V)$, i.e,

\begin{align*}
i^* f^*(G) = (f \circ i)^* (G) = (\restr{f}{V})^*G = G_{\Dqc(V)}
\end{align*}
where $G_{\Dqc(V)}$ is the restriction of $G$ to $\Dqc(V)$.

Notice this remains true if we replace $V$ by any separated open subscheme of $V$, a fact that will be used soon.

Next, one proceeds with the case of some quasicompact, quasiseparated open subscheme, say $U \subset X$, by induction on the number of affines covering $U$. The case where $U$ can be covered by only one affine is exactly the case above. So we may assume $U$ may be covered by $n$ affines and the property is true for any quasicompact, quasiseparated subscheme covered by up to $n-1$ affines.

Let $V$ be some open affine from the cover of $U$ and let $W$ be the union of the other $n-1$ affines, i.e. $U = V \cup W$. Let $Z_{V} \subset V$ be the closed subscheme of non separated points of $V$. Let $L = V \cap W$. We have three morphisms $i: V \hookrightarrow U$, $j: W \hookrightarrow U$ and $k: L \hookrightarrow U$.


Let $F \in \Dqc(U)$ and consider the square 

\[
\begin{tikzcd}
 Hom_{\Dqc(U)}(f^*G, F) \arrow{r} \arrow{d} &  Hom_{\Dqc(W)}(j^*f^*G, j^*(F)) \arrow{d}  \\
 Hom_{\Dqc(V)}(i^*f^*G, i^*(F)) \arrow{r} &  Hom_{\Dqc(L)}(k^*f^*G,k^*(F))  
\end{tikzcd}
\]
in the derived category, where an abuse of notation was used for $f = \restr{f}{U}$.

The goal is to show that the restriction of $f^*G$ to each category is a compact generator. By the induction hypothesis, the restrictions of $f^*G$ to $W, V$ are generators in each respective derived category.

For the purpose of the proof, one may assume that $W \cap Z_V = \emptyset$, as one may consider the complement of $(Z_V)^c = U \backslash Z_V$ and define $\hat{W} = W \cap (Z_V)^c$. Indeed, with $\hat{j}:\hat{W} \hookrightarrow W$ the open immersion, it suffices to show that the further restriction $\hat{j}^*j^*f^*G$ is again a generator. Notice that for any $H \in \Dqc(\hat{W})$ such that $\hat{j}_* H = 0$, one has that $H = 0$, as $(\hat{j}_*H)(U) = H(U \cap \hat{W})$. So $ 0 = Hom( \hat{j}^*f^*G, H) = Hom( f^*G, \hat{j}_* H)$ implies that $\hat{j}_*H = 0$, which implies that $H=0$. Hence, even though $\hat{W}$ may be covered by more than $n-1$ affines, one may replace $W$ for $\hat{W}$ and still get the same square of $Hom$s as before, with the restriction of $f^*G$ being a compact generator for $\hat{W}$.

Therefore, without loss of generality assume that $U = V \cup W$, with $W \cap Z_V = \emptyset$. In particular, $L \cap Z_V = \emptyset$. Now $L$ is an open separated subscheme of the affine $V$, hence the restriction of $f^*G$ to $L$ is again a generator. Assume that $Hom_{\Dqc(U)}(f^*G, F) = 0$. Then, the proof will follow if $F = 0$.

By adjunction, one have that $f_*F[n] = 0$ for all $n \in \Z$. Consider the localization sequence  

\begin{align*}
\Dqc(U \ \textbf{on} \ Z) \rightarrow \Dqc(U) \rightarrow \Dqc(L)
\end{align*}
in the derived category over $k : L \hookrightarrow U$ with $Z$ the complement of $L$, which induces the triangle

\begin{align*}
M \rightarrow F \rightarrow k_* k^* F.
\end{align*}

Applying $f_*$, one obtains the triangle

\begin{align*}
f_*M \rightarrow f_*F \rightarrow f_* k_* k^* F.
\end{align*}

By the hypothesis, the middle term is zero, and hence $f_*M[1] \simeq f_* k_* k^* F$. The claim is that $k^*F = 0$. If that was not the case, the support of $k^*F$ would not be empty, which would imply the existence of a point $p \in L$  such that $p \in Supph(k^*F)$. Let $U_p \subset L$ be some neighborhood of $p$ that satisfies the condition for $p$ in $Supph(k^*F)$. Consider the diagram


\[
\begin{tikzcd}
 U_p \arrow[hookrightarrow]{r} \arrow{d}{\simeq}[swap]{\restr{f}{U_p}} &  L \arrow[hookrightarrow]{r}{k} \arrow{d}{\simeq}[swap]{\restr{f}{L}} & U  \arrow{d}[swap]{f}  \\
 \overline{U_p} \arrow[hookrightarrow]{r} &  \overline{L} \arrow[hookrightarrow]{r}{\overline{k}} & U_{\textit{sep}}
\end{tikzcd}
\]

After a diagram chase, one sees that

\begin{align*}
f_* k_* k^* F &=  \overline{k}_* (\restr{f}{L})_* k^* F \\
 &= (\overline{k} \circ \restr{f}{L})_* k^* F.
\end{align*}

By the definition of the derived pushfoward functor, the cochain complex evaluated at $\overline{U_p} \subset U_{\textit{sep}}$ agrees with $k^*F$, i.e.,

\begin{align*}
(f_* k_* k^* F)(\overline{U_p}) &= ((\overline{k} \circ \restr{f}{L})_* k^* F )(\overline{U_p}) \\
 &= (k^* F) ((\overline{k} \circ \restr{f}{L})^{-1} (\overline{U_p})) \\
 &= k^*F({U_p}) \\
\end{align*}

So $(f_* M[1])(\overline{U_p}) = k^*F(U_p)$, but that is absurd, since $U_p \simeq f^{-1}(\overline{U_p})$ would imply that $p \notin Z$ is in the support of $M[1] \in \Dqc(U \textbf{\ on \ } Z)$. Therefore $k^* F = 0$.

Going back to the square of morphism, the top left term is $0$ by hypothesis and the bottom right is also zero, since $k^* F = 0$. Hence the whole square is zero. That means that each restriction of $F$ is zero, i.e., $j^*(F) = i^*(F) = k^*(F) = 0$
Using another square, now for $\Dqc(U)$, one may glue each restriction back to $F$. Hence $F = 0$ as desired. \end{proof}

To prove the main theorem, a standard induction argument over the covering of $X$ will be used. The following proposition will provide the induction hypothesis needed. The notation of subschemes and morphisms will follow the diagram shown in the beginning of this section.

\begin{prop}
\label{prop3.4}
Let $X$ be a quasicompact and quasiseparated scheme that admits a separator $f : X \to \Xs$.  Assume $X$ can be covered by affine subschemes $Spec(R_i)$ with each $R_i$ of finite global dimension. Moreover, let $X = U \cup V$ with $U$ and $V$ open subschemes of $X$ and assume $V$ to be affine. Consider the diagram
$$
\begin{tikzcd}[column sep=5pc]
  U \arrow{d}{u} \arrow{rd}{o} & { } \\
  X \arrow{r}{f} & \Xs
\end{tikzcd}
$$
where $u: U \to X$ is the open immersion and $o: U \to \Xs$ the induced map. 

Let $G$ be the strong generator from $\Dqc(\Xs)$. Then $u_*o^*G$ can be built in finitely many steps from $f^*G$.
\end{prop} 

\begin{proof}

First, by Proposition \ref{hopenimm}, $X$ being covered by affine subschemes of finite global dimension implies $\Xs$ also can be covered by affines with the same properties. Hence, by Theorem \ref{Amnthm} there exists a $G$ that fast generates $\Dqc (\Xs)$.

Let $G$ be the generator of $\Dqc (\Xs)$ as above. One can fit $f^*G$ into a triangle
$$
\begin{tikzcd}%
        Q \arrow[r] & f^*G \arrow[r] & u_*u^*f^*G \simeq u_*o^*G
    \end{tikzcd}%
,$$
so it suffices to show that $Q \in \Coprod _N(f^*G)$ for some $N \in \Z$.

Let $Q$ be as above. Since $Q$ vanishes on $U$, and $V$ is assumed to be affine, by Thomason-Trobaugh there exists a closed subscheme $Z \subset V$  and $P \in \Dqc( V )$ such that $Q \simeq i_*P$. 

By Theorem \ref{Amnthm} there exists a fast generator $G' \in  \Dqc( V)$. Hence, there exists $M \in \Z$, such that $ P \in \Coprod _M(G')$. 

So $Q \simeq i_*P$ is in $i_* \Coprod _M(G') \subseteq \Coprod _M(i_*G')$.

But, by Proposition \ref{prop3.2}, $i_*G'$ is a retract of $f^*j_*G'$, which again by Theorem \ref{Amn6.2} is in $\Coprod _N(f^*G)$ for some $N$. Hence $Q \in \Coprod _N(f^*G)$, proving that $u_*o^*G$ is indeed generated by $f^*G$. \end{proof}

%
%

Now we move to the main Theorem \ref{main}. The goal is to show that one may pull back a fast generator via a separator and obtain a fast generator. By Proposition \ref{propreduc}, this implies Theorem \ref{main}.

\begin{theorem}
\label{thm3.5}
Let $X$ be a quasicompact and quasiseparated scheme that admits a separator $f : X \to \Xs$ and let $G$ be a compact fast generator of $\Dqc(\Xs)$. Assume $X$ can be covered by affine subschemes $Spec(R_i)$ with each $R_i$ of finite global dimension. Then there exists an object $H$ in $\textbf{D}_{\textbf{perf}} (X)$ that fast generates $\Dqc(X)$.
\end{theorem}

\begin{proof}

First, we notice that by Lemma \ref{lem3.3}, $f^*G$ is already a compact generator. Hence, it suffices to show that it is a fast generator. We proceed by induction on the number of affines in the cover of $X$ to show that $f^*G$ is indeed a fast generator.

 The case $n=1$ means that $X$ is affine, hence separated. Therefore, the separator $f$ is an isomorphism and $f^*G = G$.
 
 Assume the theorem holds for any scheme which admits a cover by up to $n$ affines $Spec(R_i)$, each $R_i$ with finite global dimensions. Suppose that $X$ can be covered by $n+1$ affines $U_i = Spec(R_i)$, each $R_i$ with finite global dimensions, i.e., $X = \bigcup^{n+1}_{i=1} U_i $.
 
 Let $U = \bigcup^{n}_{i=1} U_i$ and $V = U_{n+1}$, so $X = U \cup V$. Assume we are in the same situation as the previous diagrams.
 
Let $G \in \Dqc (\Xs)$ be a fast generator. Since the restriction of $f$ to $U$ is also a separator, by Lemma \ref{lem3.3}, the restriction of $f^*G$ to $U$, i.e. $u^*f^*(G) = o^*(G)$ is a compact generator. By induction hypothesis, there exists $G'$ that fast generates $\Dqc (U)$. Since $G'$ is compact and is in the subcategory generated by coproducts of $o^*(G)$, without loss of generality we may take $o^*G$ to be the fast generator of $\Dqc (U)$. Now, by Proposition \ref{prop3.4}, there exist $N$ such that $u_*o^*G \in \Coprod _N(f^*G)$.
 
 In similar fashion, since $V$ is an affine open from $X$, we may take $j^*G$ as a fast generator of $\Dqc (V)$. 
 
 Using another localization sequence, one obtain the triangle 

$$
\begin{tikzcd}%
        H \arrow[r] & f^*G \arrow[r] & i_*i^*f^*G = i_*j^*G
    \end{tikzcd}%
$$ 
in $\Dqc (X)$ for $H$ not supported in $V$. That implies that there exist some $P \in \Dqc (U)$ such that $H = u_* P$. By the previous paragraph, $\Dqc (U)$ is fast generated by $o^*(G)$, which implies that $H \in Coprod_L ( u_*o^*G)$, for some $L >0$. Since $u_*o^*G \in Coprod_N(f^*G)$, one obtains that $H \in Coprod_{LN}(f^*G)$. Therefore, there exists $M  > LN > 0$ such that $i_*j^*G \in Coprod_M(f^*G)$.
 
 Let $T = U \cap V$ and $t: T \to U$ be the inclusion. Then, any object $F \in \Dqc (X)$ fits in the triangle
 
 $$
\begin{tikzcd}%
        u_*[t_*t^*u^* \Sigma^{-1} F ] \arrow[r] & F \arrow[r] & u_*[u^*F] \oplus i_*[i^*F].
    \end{tikzcd}%
$$

Thus, $F$ belongs to $[u_*( \Dqc (U)] \star [u_* \Dqc (U) \oplus i_* \Dqc(V)] $ which is contained in
\begin{align*}
Coprod _{MN}(f^*G) \star Coprod _{MN}(f^*G) = Coprod _{2MN}(f^*G)
\end{align*}

Therefore $\Dqc(X)$ is fast generated and the result follows 
\end{proof}

Theorem \ref{thm3.5} together with Remark \ref{rmk1.3} prove the main Theorem \ref{main}, restated below

\begin{theorem*}  Let $X$ be a quasicompact, quasiseparated scheme that admits a separator. Then $\textbf{D}_{\text{perf}}(X)$ is regular if, and only if, $X$ can be covered by open affine subschemes $Spec(R_i)$ with each $R_i$ of finite global dimension.
\end{theorem*}

\newpage


\end{document}